\documentclass[12pt,a4paper]{article}
\usepackage{graphics}
\usepackage{graphicx}
\usepackage{caption}
\usepackage{floatrow}
\usepackage{float}
\usepackage{pifont}
\usepackage{subcaption}
\usepackage{epsfig}
\usepackage[mathscr]{euscript}
\usepackage{mathrsfs}
\usepackage{amsmath,amsthm,amsfonts}
\usepackage{amssymb}
\usepackage{enumerate}
\usepackage{hyperref}
\usepackage{blkarray}
\usepackage{multirow}

\newtheorem{theorem}{Theorem}[section]

\newtheorem{proposition}[theorem]{Proposition}
\newtheorem{prop}[theorem]{Proposition}
\newtheorem{definition}[theorem]{Definition}
\newtheorem{conjecture}[theorem]{Conjecture}
\newtheorem{lemma}[theorem]{Lemma}
\newtheorem{corollary}[theorem]{Corollary}
\newtheorem{remark}[theorem]{Remark}
\newtheorem{example}[theorem]{Example}

\newcommand{\N}{\mathbf{N}}

\newcommand{\ho}{\mathfrak{h}}
\newcommand{\Go}{\mathfrak{G}}
\newcommand{\go}{\mathfrak{g}}

\newcommand{\fo}{\mathfrak{f}}

\begin{document}
\title{On graphs with smallest eigenvalue at least $-3$ and their lattices}

\author{Jack H. Koolen$^{\dag\S}$, Jae Young Yang$^\ddag$, Qianqian Yang$^\dag$
\\ \\
\small $^\dag$School of Mathematical Sciences,\\
\small University of Science and Technology of China, \\
\small 96 Jinzhai Road, Hefei, 230026, Anhui, PR China\\
\small $^\S$Wen-Tsun Wu Key Laboratory of CAS,\\
\small 96 Jinzhai Road, Hefei, 230026, Anhui, PR China\\
\small $^\ddag$School of Mathematical Sciences,\\
\small Anhui University, \\
\small 111 Jiulong Road, Hefei, 230039, Anhui, PR China\vspace{4pt}\\
\small {\tt e-mail: koolen@ustc.edu.cn, piez@naver.com, xuanxue@mail.ustc.edu.cn}\vspace{-3pt}
}
\maketitle
\date{}
\vspace{-24pt}
\begin{abstract}
In this paper, we show that a connected graph with smallest eigenvalue at least  $-3$ and large enough minimal degree is $2$-integrable. This result generalizes a 1977 result of Hoffman for connected graphs with smallest eigenvalue at least $-2$.
\end{abstract}

\textbf{Keywords:} $s$-integrability, lattice, root lattice, Hoffman graph, eigenvalue

\textbf{AMS classification:} 05C50, 11H99

\section{Introduction}
(For undefined notions, see next section) Graphs with smallest eigenvalue at least $-2$ are characterized by Cameron et al. \cite{Cameron} in 1976. They showed that

\begin{theorem}\label{cameron}
Let $G$ be a connected graph with smallest eigenvalue at least $-2$. Then, either $G$ is a generalized line graph, or $G$ has at most $36$ vertices.
\end{theorem}

Concerning graphs with large minimal degree, Hoffman \cite{Hoff1977} showed:

\begin{theorem}\label{hfthm}
For any real number $\lambda\in (-1-\sqrt{2},-2]$, there exists a constant $f(\lambda)$ such that if a connected graph $G$ has smallest eigenvalue at
least $\lambda$ and minimal degree at least $f(\lambda)$, then $G$ is a generalized line graph.
\end{theorem}

In this paper, we are going to generalize Theorem \ref{hfthm} for graphs with smallest eigenvalue at least $-3$. In order to state our main result, we need to introduce the following terminology. Let $G$ be a graph with smallest eigenvalue $\lambda_{\min}$. Define the positive semi-definite matrix $B$ by $A - \lfloor\lambda_{\min} \rfloor I$, where $A$ is the adjacency matrix of $G$. For a positive integer $s$, we say that the graph $G$ is \emph{$s$-integrable} if there exists a matrix $N$ such that the equality $B = N^TN$ holds and the matrix $\sqrt{s}N$ has only integral entries.

Our main result is as follows.
\begin{theorem}\label{mainthm}
There exists a positive integer $K$ such that if a connected graph $G$ has smallest eigenvalue at least $-3$ and minimal degree at least $K$, then $G$ is $2$-integrable.
\end{theorem}
The above result partially solves a problem in \cite[p111]{drg}.

\begin{remark}\label{rem1}
\begin{enumerate}
\item Theorem \ref{mainthm} can be seen as a generalization of Theorem \ref{hfthm}, as the generalized line graphs are exactly the 1-integrable graphs with smallest eigenvalue at least $-2$.

\item  As part of the proof of Theorem \ref{mainthm}, we can show that there exists a real number $\varepsilon<-3$, such that for any real number $\lambda\in(\varepsilon,-3]$, there exists a constant $f(\lambda)$, such that if a graph $G$ has smallest eigenvalue at least $\lambda$ and minimal degree at least $f(\lambda)$, then the smallest eigenvalue of $G$ is at least $-3$.

\item In Example \ref{E6}, we will give a family of non $1$-integrable graphs with smallest eigenvalue at least $-3$ and unbounded minimal degree. This shows that we cannot improve the result in Theorem \ref{mainthm} from $2$-integrability to $1$-integrability.

\item Denote by $\overline{McL}$ the complement of the McLaughlin graph. Note that $\overline{McL}$ is the unique strongly regular graph with parameters $(275, 162, 105, 81)$, as the McLaughlin graph is uniquely determined by its parameters (See \cite{Mclaughlin}). In \cite{KRY2}, Koolen et al. showed that $\overline{McL}$ is not $2$-integrable. Let $H$ be the $3$-point cone over $\overline{McL}$, that is, $H$ is obtained by adding $3$ new vertices to $\overline{McL}$ and joining these three vertices to all the vertices in $\overline{McL}$ and to each other. Note that the graph $H$ has minimal degree $165$, smallest eigenvalue $-3$ and is not $2$-integrable, as $\overline{McL}$ is not $2$-integrable. This shows that the constant $K$ in Theorem \ref{mainthm} is at least $166$.  The graph $H$ was found by Koolen and Munemasa and they also showed that this is a maximal graph with smallest eigenvalue at least $-3$, that is, there does not exist a connected graph with smallest eigenvalue at least $-3$ which contains $H$ as a proper induced subgraph.

\item There exist families of connected non $2$-integrable graphs with smallest eigenvalue at least $-3$ and unbounded number of vertices, as shown by Koolen and Munemasa. For example, let $H^\prime$ be the cone over $\overline{McL}$ and join the complete graph $K_m$ to the cone vertex of $H^\prime$ to obtain the connected graph $K(m)$. Then the connected graph $K(m)$ with $m+276$ vertices has smallest eigenvalue $-3$ and is not $2$-integrable.

\begin{figure}[h]
\begin{center}
\includegraphics[scale=1]{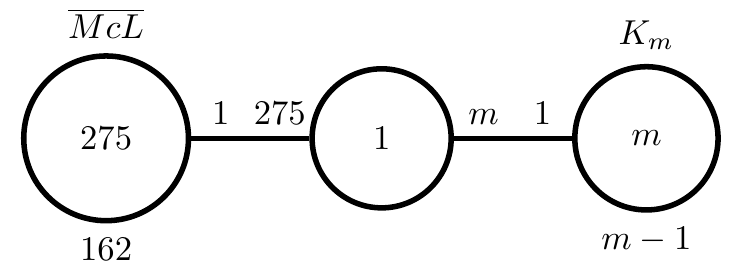}
\end{center}
  \caption{The graph $K(m)$}
\end{figure}
\end{enumerate}
\end{remark}

Now we look at integral lattices with minimal vectors of norm $3$.

Let $\Lambda$ be an (finite-dimensional) integral lattice whose minimal vectors have norm $3$, and let $\Lambda^\prime$ be the sublattice generated by the minimal vectors of $\Lambda$. We denote by $U$ the set of the minimal vectors of $\Lambda$, that is, $U=\{\mathbf{u}\in\Lambda\mid\| \mathbf{u}\|^2=3\}$. Assume that $U^\prime$ is a maximal subset of $U$ such that for any two distinct vectors $\mathbf{u}_1$ and $\mathbf{u}_2$ in $U^\prime$, $\mathbf{u}_1\neq -\mathbf{u}_2$, and $L$ is the Gram matrix of $U^\prime$. Since any two distinct vectors $\mathbf{u}_1$ and $\mathbf{u}_2$ in $U^\prime$ have inner product $0,1$ or $-1$, we find that $L=3I+A$, where $A$ is the adjacency matrix of some signed graph $\widehat{G}$. This concludes that $A$ (and $\widehat{G}$) has smallest eigenvalue at least $-3$. Therefore, in order to understand lattices whose minimal vectors have norm $3$, we need to understand signed graphs with smallest eigenvalue at least $-3$. In this paper, we will look at connected graphs with smallest eigenvalue at least $-3$. We think that graphs and signed graphs behave very similar with respect to the smallest eigenvalue.

Related work in terms of line systems with angles $\alpha$ satisfying as $\alpha\in\{0,\pm\frac{1}{3}\}$ has been  done by Shult and Yanushka \cite{Ernest}. Bachoc, Nebe and Venkov \cite{odd unimodular lattices} studied lattices with minimal vectors of norm $t$, where $t$ is an odd integer.

To show our main result, we need the theory of Hoffman graphs as introduced by Woo and Neumaier \cite{Woo}. In Section \ref{preliminary}, we will give a collection of known results on Hoffman graphs which will be used to prove our main result.

\section{Definitions and preliminaries}\label{preliminary}

\subsection{Lattices and graphs}

A \emph{lattice} $\Lambda$ in $\mathbb{R}^n$ is a set of vectors in $\mathbb{R}^n$ which is closed under addition and subtraction. We say that a lattice is generated by a set $X$ of $\mathbb{R}^n$ if $\Lambda = \{
\sum_{\mathbf{x} \in X} \alpha_\mathbf{x} \mathbf{x} \mid \alpha_\mathbf{x} \in \mathbb{Z}$ for all $\mathbf{x} \in X\}$, and we write $\Lambda = \langle X \rangle$. A lattice $\Lambda$ is called \emph{integral} if the inner product of any two vectors in $\Lambda$ is integral. The minimal vectors of an integral lattice are the vectors with minimal positive norm. We say an integral lattice has minimal norm $t$ if its minimal vectors have norm $t$.  An integral lattice is called {\it irreducible} if it is not a direct sum of proper sublattices. An integral lattice $\Lambda$ is called a \emph{root lattice} if $\Lambda$ is generated by a set $X$ of vectors of norm $2$, i.e. $\langle \mathbf{x},\mathbf{x} \rangle=2$ for all $\mathbf{x} \in X$. The vectors of norm $2$ in a root lattice $\Lambda$ are called the {\it roots} of $\Lambda$.

\begin{theorem}\label{root} (\cite{Witt}) Every irreducible root lattice is isomorphic to one of the following:

\begin{enumerate}

\item $A_n = \{ \mathbf{x} \in \mathbb{Z}^{n+1}\mid \sum \mathbf{x}_i = 0\}$ for $n \geq 1$,

\item $D_n = \{ \mathbf{x} \in \mathbb{Z}^{n}\mid \sum \mathbf{x}_i ~\text{is even}\}$ for $n \geq 4$,

\item $E_8 = D_8 \cup (\mathbf{c}+D_8)$, where $\mathbf{c}= \frac{1}{2}(1,1,1,1,1,1,1,1)$,

\item $E_7 = \{ \mathbf{x} \in E_8 \mid \sum \mathbf{x}_i = 0 \}$,

\item $E_6 = \{ \mathbf{x} \in E_7 \mid \mathbf{x}_7 + \mathbf{x}_8 = 0\}$.

\end{enumerate}

\end{theorem}

For a positive integer $s$, an integral lattice $\Lambda \subset \mathbb{R}^n$ is called {\it $s$-integrable} if $\sqrt{s}\Lambda$ can be embedded in the standard lattice $\mathbb{Z}^k$ for some $k \geq n$. In other words, an integral lattice $\Lambda$ is $s$-integrable if and only if every vector in $\Lambda$ can be described by the form $\frac{1}{\sqrt{s}}(x_1, \dots, x_k)$ with all $x_i \in \mathbb{Z}$ in $\mathbb{R}^k$ for some $k \geq n$ simultaneously.

\begin{remark}\label{root2}
In Theorem \ref{root}, $A_n$ and $D_n$ are $1$-integrable. However, $E_6$, $E_7$, and $E_8$ are not $1$-integrable \cite[Theorem 2]{low dimension V}. To show that they are $2$-integrable, we will consider another representation \cite{code} for them.
Define $8$ vectors ${\mathbf{u}^i}$ for $i=1, \dots, 8$ as follows:

$$\mathbf{u}^1 = \frac{1}{\sqrt{2}}(0,1,1,0,1,0,0,1),~~~~~~$$
$$\mathbf{u}^2 = \frac{1}{\sqrt{2}}(0,-1,0,1,-1,1,0,0),~~$$
$$\mathbf{u}^3 = \frac{1}{\sqrt{2}}(0,0,-1,0,1,-1,1,0),~~$$
$$\mathbf{u}^4 = \frac{1}{\sqrt{2}}(1,0,0,-1,0,1,-1,0),~~$$
$$\mathbf{u}^5 = \frac{1}{\sqrt{2}}(-1,1,0,0,-1,0,1,0),~~$$
$$\mathbf{u}^6 = \frac{1}{\sqrt{2}}(1,-1,1,0,0,-1,0,0),~~$$
$$\mathbf{u}^7 = \frac{1}{\sqrt{2}}(0,1,-1,1,0,0,-1,0),~~$$
$$\mathbf{u}^8 = \frac{1}{\sqrt{2}}(-1,-1,0,0,1,0,-1,0).$$

Then, it is known that
$$E_8 \simeq \langle ~\mathbf{u}^i \mid i = 1,\dots, 8~ \rangle,$$
$$E_7 \simeq \langle ~\mathbf{u}^i \mid i = 2,\dots, 8~ \rangle,$$
$$E_6 \simeq \langle ~\mathbf{u}^i \mid i = 3,\dots, 8~ \rangle.$$

Hence, $E_6$, $E_7$, and $E_8$ are $2$-integrable.

\end{remark}

Let $G$ be a connected graph with adjacency matrix $A(G)$ and smallest eigenvalue $\lambda_{\min}(G)$. For $t = -\lfloor \lambda_{\min} (G) \rfloor$, the matrix $B = A+tI$ is positive semidefinite, and  hence there exists a matrix $N$ such that $B = N^T N$. Let $\Lambda(G)$ be the integral lattice generated by the columns of the matrix $N$. Then we find that $\Lambda(G)$ is an integral lattice generated by vectors of norm $t$. Note that the graph $G$ is $s$-integrable if and only if the integral lattice $\Lambda(G)$ is an $s$-integrable lattice.

\subsection{Hoffman graphs}
Now we give definitions and preliminaries of Hoffman graphs. For more details, see \cite{HJAT}, \cite{KYY} and \cite{Woo}.
\begin{definition} A Hoffman graph $\mathfrak{h}$ is a pair $(H,\ell)$ where $H=(V,E)$ is a graph and $\ell:V \rightarrow \{f,s\}$ is a labeling map satisfying the following conditions:
\begin{enumerate}
\item Every vertex with label $f$ is adjacent to at least one vertex with label $s$;
\item Vertices with label $f$ are pairwise non-adjacent.
\end{enumerate}
\end{definition}
We call a vertex with label $s$ a \emph{slim vertex}, and a vertex with label $f$ a \emph{fat vertex}. We denote
by $V_s(\mathfrak{h})$ (resp. $V_f(\mathfrak{h})$) the set of slim (resp. fat) vertices of $\mathfrak{h}$.

\vspace{0.1cm}
For a vertex $x$ of $\mathfrak{h}$, we define $N_{\mathfrak{h}}^{s}(x)$ (resp. $N_{\mathfrak{h}}^{f}(x)$) the set of slim (resp. fat) neighbors of $x$ in $\mathfrak{h}$.
If every slim vertex of the Hoffman graph $\mathfrak{h}$ has a fat neighbor, then we call $\mathfrak{h}$ \emph{fat}. 

\vspace{0.1cm}
The \emph{slim graph} of the Hoffman graph $\mathfrak{h}$ is the subgraph of $H$ induced on $V_s(\mathfrak{h})$. Note that any graph can be considered as a Hoffman graph with only slim vertices, and vice versa. We will not distinguish between Hoffman graphs with only slim vertices and graphs.

\vspace{0.1cm}
A Hoffman graph $\mathfrak{h}_1= (H_1, \ell_1)$ is called an \emph{(proper) induced Hoffman subgraph} of $\mathfrak{h}=(H, \ell)$, if $H_1$ is an (proper) induced subgraph of $H$ and $\ell_1(x) = \ell(x)$ holds for all vertices $x$ of $H_1$.

\vspace{0.1cm}
Let $W$ be a subset of $V_s(\mathfrak{h})$. An \emph{induced Hoffman subgraph of $\mathfrak{h}$ generated by $W$}, denoted by $\langle W\rangle_{\mathfrak{h}}$, is the Hoffman subgraph of $\mathfrak{h}$ induced by $W \cup\{f\in V_f(\mathfrak{h})~|f \sim w \text{ for some }w\in W \}$.

\vspace{0.1cm}

\begin{definition} For a Hoffman graph $\mathfrak{h}=(H,\ell)$, there exists a matrix $C$ such that the adjacency matrix $A$ of $H$ satisfies
\begin{eqnarray*}
A=\left(
\begin{array}{cc}
A_s  & C\\
C^{T}  & O
\end{array}
\right),
\end{eqnarray*}
where $A_s$ is the adjacency matrix of the slim graph of $\mathfrak{h}$. The \emph{special matrix} $Sp(\mathfrak{h})$ of $\mathfrak{h}$ is the real symmetric matrix $Sp(\mathfrak{h}):=A_s-CC^{T}.$ The eigenvalues of $\mathfrak{h}$ are the eigenvalues of $Sp(\mathfrak{h})$, and the smallest eigenvalue of $\mathfrak{h}$ is denoted by  $\lambda_{\min}(\mathfrak{h})$.
\end{definition}

\begin{lemma}(\cite[Lemma 3.4]{Woo})\label{fatnbr}
Let $\ho$ be a Hoffman graph and let $x_i$ and $x_j$ be two distinct slim vertices of $\ho$. The special matrix $Sp(\mathfrak{h})$ has diagonal entries

$$Sp(\mathfrak{h})_{x_i,x_i} = - |N_\ho^f (x_i)|$$
and off-diagonal entries
$$Sp(\mathfrak{h})_{x_i,x_j} = (A_s)_{x_i,x_j} -|N_\ho^f (x_i) \cap N_\ho^f (x_j)|.$$
\end{lemma}

For the smallest eigenvalues of Hoffman graphs and their induced Hoffman subgraphs, Woo and Neumaier showed the following inequality.
\begin{lemma}(\cite[Corollary 3.3]{Woo})\label{hoff}
If $\mathfrak{h}_1$ is an induced Hoffman subgraph of a Hoffman graph $\mathfrak{h}$, then $\lambda_{\min}(\mathfrak{h}_1)\geq\lambda_{\min}(\mathfrak{h})$ holds.
\end{lemma}
\subsubsection{Two canonical fat Hoffman graphs}\label{twokindsofhoffman}
Given a graph $H$, there are two canonical fat Hoffman graphs $\mathfrak{p}(H)$ and $\mathfrak{q}(H)$ with $H$ as their slim graphs. The fat Hoffman graph $\mathfrak{p}(H)$ has $H$ as its slim graph and to each slim vertex, a fat vertex is attached such that any two distinct slim vertices do not have a common fat neighbor. The fat Hoffman graph $\mathfrak{q}(H)$ has $H$ as its slim graph and one fat vertex that is attached to each slim vertex. Note that $Sp(\mathfrak{p}(H))=A(H)-I$ and $Sp(\mathfrak{q}(H))=-A(\overline{H})-I$, where $A(H)$ and $A(\overline{H})$ are the adjacency matrices of $H$ and its complement $\overline{H}$, respectively. Then we have
\begin{lemma}\label{smallesteigenvalueoftwohoff}
Let $H$ be a graph with $\overline{H}$ as its complement. Denote by $\lambda_{\min}(H)$ and $\lambda_{\max}(\overline{H})$ the smallest eigenvalue of $H$ and the largest eigenvalue of $\overline{H}$, respectively. Then $\lambda_{\min}(\mathfrak{p}(H))=-1+\lambda_{\min}(H)$ and $\lambda_{\min}(\mathfrak{q}(H))=-1-\lambda_{\max}(\overline{H})$.
\end{lemma}
\begin{proof}It follows immediately from their special matrices.
\end{proof}

\subsection{The Hoffman-Ostrowski theorem}\label{sectionostrowski}
In this subsection, we introduce a result of Hoffman and Ostrowski. In order to state the result, we need to introduce the following notations.  Suppose $\ho$ is a Hoffman graph and $\{f_1, \dots, f_r\}$ is a subset of $V_f(\mathfrak{h})$. Let $\go^{n_1, \ldots, n_r}(\ho)$ be the Hoffman graph obtained from $\ho$ by replacing the fat vertex $f_i$ by a slim $n_i$-clique $K^{f_i}$, and joining all the neighbors of $f_i$ (in $\ho$) with all the vertices of $K^{f_i}$  for all $i$. We will write $G(\ho,n)$ for the graph $\go^{n_1, \ldots, n_r}(\ho)$, when $V_f(\ho) = \{ f_1, f_2, \ldots, f_r\}$ and $n_1 =n_2 = \dots = n_r = n$. With the above notations, we can now state the result of Hoffman and Ostrowski.  For a proof of it, see \cite[Theorem 2.14]{HJAT}. We will give a closely related result in the next section.

\begin{theorem}\label{Ostrowski} Suppose $\ho$ is a Hoffman graph with fat vertices $f_1, f_2, \dots, f_r $.Then
\begin{center}$\lambda_{\min}(\go^{n_1,\dots, n_r}(\ho))\geq \lambda_{\min}(\ho)$,\end{center}
and
\begin{displaymath}\lim_{n_1,\dots,n_r \rightarrow \infty}\lambda_{\min}(\go^{n_1,\dots, n_r}(\ho))= \lambda_{\min}(\ho).\end{displaymath}
\end{theorem}

\subsection{Lattices and Hoffman graphs}
Let $\mathfrak{h}=(H,\ell)$ be a Hoffman graph with smallest eigenvalue $\lambda_{\min}(\mathfrak{h})$, and let $t$ be a real number satisfying $\lambda_{\min}(\mathfrak{h})\geq -t$. Then the symmetric matrix $Sp(\mathfrak{h})+tI$ is positive semidefinite. Therefore, there exists a real matrix $N$ such that $Sp(\mathfrak{h})+tI=N^TN$ holds. We denote by $\Lambda^{red}(\frak{h},t)$ the lattice generated by the columns of $N$. Note that the isomorphism class of $\Lambda^{red}(\frak{h},t)$ depends only on $t$ and $Sp(\mathfrak{h})$, and is independent of $N$.

For a real number $p$, we define $L(p)$ to be a diagonal matrix indexed by the vertex set $V_s(\mathfrak{h})\cup V_f(\mathfrak{h})$ such that $L(p)_{x,x}=p$ if $x\in V_s(\mathfrak{h})$ and $L(p)_{x,x}=1$ if $x\in V_f(\mathfrak{h})$. Note that the matrix $A(H)+L(t)$ is also positive semidefinite (see \cite[Theorem 2.8]{HJAT}), where $A(H)$ is the adjacency matrix of $H$.  Hence, there exists a real matrix $\widetilde{N}$ such that $A(H)+L(t)=\widetilde{N}^T\widetilde{N}$ holds. We denote by $\Lambda(\frak{h},t)$ the lattice generated by the columns of $\widetilde{N}$. Note that the isomorphism class of $\Lambda(\frak{h},t)$ depends only on $t$ and $\mathfrak{h}$, and is independent of $\widetilde{N}$.

It is fairly easy to see that the matrices $N$ and $\widetilde{N}$ of above are closely related. Using this relation, we obtain the following lemma. For the convenience of the readers we give a proof.
\begin{lemma}\label{twolattices}
Let $\mathfrak{h}=(H,\ell)$ be a Hoffman graph with smallest eigenvalue $\lambda_{\min}(\mathfrak{h})$, and let $t$ be an integer satisfying $\lambda_{\min}(\mathfrak{h})\geq -t$. The integral lattice $\Lambda^{red}(\frak{h},t)$ is $s$-integrable if and only if the integral lattice $\Lambda(\frak{h},t)$ is $s$-integrable.
\end{lemma}
\begin{proof}
Suppose that the lattice $\Lambda^{red}(\frak{h},t)$ is $s$-integrable. It implies that there exists a matrix $N$ such that $Sp(\mathfrak{h})+tI=N^TN$, where each entry of $N$ is in $\frac{1}{\sqrt{s}}\mathbb{Z}$. We denote by $N^\prime$ the fat-slim vertex incidence matrix with respect to the set of fat vertices of $\mathfrak{h}$, that is, $(N^\prime)_{f,x}=\left\{
  \begin{array}{ll}
    1 &\text{ if $x$ is adjacent to $f$}, \\
    0 &\text{ otherwise}, \\
  \end{array}
\right.$
for $x\in V_s(\mathfrak{h})$ and $f\in V_f(\mathfrak{h})$. The matrix $\widetilde{N}:=\begin{pmatrix}N & O\\N^\prime&I\end{pmatrix}$ satisfies the equation $\widetilde{N}^T\widetilde{N}=A(H)+L(t)$ and this shows that the lattice $\Lambda(\frak{h},t)$ is $s$-integrable.

Conversely, suppose that the lattice $\Lambda(\frak{h},t)$ is $s$-integrable. Then there exists a matrix $\widetilde{N}$ such that $A(H)+L(t)=\widetilde{N}^T\widetilde{N}$ holds, where each  entry of $\widetilde{N}$ is in $\frac{1}{\sqrt{s}}\mathbb{Z}$. The matrix $N$ satisfying $N^TN=Sp(\mathfrak{h})+tI$ can be obtained, if for any slim vertex $x$ of $\mathfrak{h}$, we define the column $N_x$ of $N$ indexed by $x$ as follows:
$$N_x:=\widetilde{N}_x-\sum_{\substack{f\sim x\\f\in V_f(\mathfrak{h})}}\widetilde{N}_f,$$
where $\widetilde{N}_x$ and $\widetilde{N}_f$ are the columns of $\widetilde{N}$ indexed by vertices $x$ and $f$, respectively. Note that all entries of $N$ are in $\frac{1}{\sqrt{s}}\mathbb{Z}$ and this shows that the lattice $\Lambda^{red}(\frak{h},t)$ is $s$-integrable.
\end{proof}

The following corollary immediately follows from Lemma \ref{twolattices}.

\begin{corollary}\label{graphandhoffman}
Let $\mathfrak{h}=(H,\ell)$ be a Hoffman graph with smallest eigenvalue $\lambda_{\min}(\mathfrak{h})$, and let $H^\prime$ be the slim graph of $\mathfrak{h}$ with smallest eigenvalue $\lambda_{\min}(H^\prime)$. Suppose $\lfloor\lambda_{\min}(\mathfrak{h})\rfloor=\lfloor\lambda_{\min}(H^\prime)\rfloor$ holds. Then the graph $H^\prime$ is $s$-integrable if the integral lattice $\Lambda(\frak{h},-\lfloor\lambda_{\min}(\mathfrak{h})\rfloor)$ is $s$-integrable.
\end{corollary}

Another result for $s$-integrability of Hoffman graphs is given as follows.

\begin{theorem}\label{equivrepresentation}Let $\mathfrak{h}$ be a fat Hoffman graph with smallest eigenvalue $\lambda_{\min}(\mathfrak{h})$ satisfying the following:
$$-3\leq\lambda_{\min}(\mathfrak{h})<-2.$$
Then the integral lattice $\Lambda^{red}(\frak{h},3)$ is $1$-integrable if and only if there exists a positive constant $D$ such that the graph $G(\mathfrak{h},n)$ is $1$-integrable for all $n\geq D$.
\end{theorem}
\begin{proof}From Theorem \ref{Ostrowski} we have that both $\lambda_{\min}(G(\mathfrak{h},n))\geq\lambda_{\min}(\mathfrak{h})$ and $\lim\limits_{n \rightarrow \infty}\lambda_{\min}(G(\mathfrak{h},n))=\lambda_{\min}(\mathfrak{h})$ hold. It follows that there exists a constant $D^\prime$ such that $-3\leq\lambda_{\min}(G(\mathfrak{h},n))<-2$ holds for all $n\geq D^\prime$.

Assume $V_s(\mathfrak{h})=\{x_1,\ldots,x_m\}$ and $V(K^f)=\{y_1^f,\ldots, y_n^f\}$ for any $f\in V_f(\mathfrak{h})$, where $K^f$ is the $n$-clique as defined in Subsection \ref{sectionostrowski}.

If the integral lattice $\Lambda^{red}(\frak{h},3)$ is $1$-integrable, then there exists an integral matrix $N^\prime$ such that $Sp(\mathfrak{h})+3I=(N^\prime)^TN^\prime$ holds. Let $D:=D^\prime$ be the constant. In order to show that $G(\mathfrak{h},n)$ is $1$-integrable for all $n\geq D$, it is sufficient to show that there exists an integral matrix $N$ such that $A(G(\mathfrak{h},n))+3I=N^TN$ holds for all $n\geq D$. For any vertices $x_i$ and $y_j^f$ of $G(\mathfrak{h},n)$, we define the columns $N_{x_i}$ and $N_{y_j^f}$ of $N$ indexed by the vertices $x_i$ and $y_j^f$ as follows:
$$N_{x_i}:=N^\prime_{x_i}+\sum\limits_{\substack{f\sim {x_i}\\f\in V_f(\mathfrak{h})}}\mathbf{e}_f,~N_{y^f_j}:=\mathbf{e}_f+\mathbf{e}_{y^f_{j,1}}+\mathbf{e}_{y^f_{j,2}},$$
where $\{\mathbf{e}_f,\mathbf{e}_{y^f_{j,1}},\mathbf{e}_{y^f_{j,2}}\mid f\in V_f(\mathfrak{h}),j=1,\ldots,n\}$ is a set of orthonormal integral vectors which are orthogonal to the matrix $N^\prime$. It is easy to check that the matrix $N$ satisfies $A(G(\mathfrak{h},n))+3I=N^TN$.

Conversely, assume that the graph $G(\mathfrak{h},n)$ is $1$-integrable for all $n\geq D$. Let $n^\prime:=\max\{D,D^\prime,20\}$ be a positive integer. Then there exists an integral matrix $N$ such that $A(G(\mathfrak{h},n^\prime))+3I=N^TN$ holds. For any vertex $x$ of $G(\mathfrak{h},n^\prime)$, we denote by $N_x$ the column of $N$ indexed by $x$, and for the column vector $N_x$, we denote by $\underline{N}_x$ the support of $N_x$, that is, the set $\{j\mid (N_x)_j\neq0\}$. Note that the size of $\underline{N}_x$ is $3$. Now for any vertex $x_i$ of $G(\mathfrak{h},n^\prime)$ and any fat vertex $f$ adjacent to $x_i$ in $\mathfrak{h}$, we have $\underline{N}_{x_i}\cap\underline{N}_{y^f_j}\neq\emptyset$ for $j=1,\ldots,n^\prime$, as $x_i$ and $y_j^f$ are adjacent in $G(\mathfrak{h},n^\prime)$. Moreover, as $n\geq20$, we can find four vertices $y_{j_1}^f,~ y_{j_2}^f,~y_{j_3}^f$ and $y_{j_4}^f$ satisfying the following conditions:
    \begin{itemize}
    \item there exists an integer $r(f)$ such that $r(f)=\cap_{p=1}^4\underline{N}_{y^f_{j_p}}\cap\underline{N}_{x_i}$;
    \item $(N_{x_i})_{r(f)}=(N_{y^f_{j_p}})_{r(f)}\in\{1,-1\}$ for all $1\leq p\leq4$;
    \item $|\underline{N}_{y^f_{j_p}}\cap\underline{N}_{y^f_{j_q}}|=1$, for all $1\leq p<q\leq4$.
    \end{itemize}
     Then, for any $t\neq i$, if $x_t$ is adjacent to $f$ in $\mathfrak{h}$, then we have $r(f)\in\underline{N}_{x_t}$ and $(N_{x_t})_{r(f)}=(N_{x_i})_{r(f)}$, as $x_t$ is adjacent to all of $y_{j_1}^f,~ y_{j_2}^f,~y_{j_3}^f$ and $y_{j_4}^f$ and the size of $\underline{N}_{x_t}$ is $3$. Moreover, if $x_t$ is not adjacent to $f$ in $\mathfrak{h}$, then we have $r(f)\not\in\underline{N}_{x_t}$, as $x_t$ is adjacent to none of $y_{j_1}^f,~ y_{j_2}^f,~y_{j_3}^f$ and $y_{j_4}^f$ and the size of $\underline{N}_{x_t}$ is $3$. Now considering the submatrix of $N$ indexed by the vertices of $V_s(\mathfrak{h})$, we change its $(r(f),x_i)^{\rm{th}}$ position from non-zero to zero for $i=1,\ldots, m$ and obtain a new matrix $N^\prime$. It is easy to check that the integral matrix $N^\prime$ satisfies $Sp(\mathfrak{h})+3I=(N^\prime)^TN^\prime$. This shows that $\Lambda^{red}(\frak{h},3)$ is $1$-integrable.
\end{proof}

Using the fat Hoffman graph $\mathfrak{p}(H)$, as defined in Subsection \ref{twokindsofhoffman}, we obtain the following corollary.

\begin{corollary}\label{twokindsofgraphs}
Assume that $H$ is a non-complete graph with smallest eigenvalue at least $-2$. Then the graph $H$ is $1$-integrable if and only if there exists a constant $D$ such that the graph $G(\mathfrak{p}(H),n)$ is $1$-integrable for all $n\geq D$.
\end{corollary}
\begin{proof}
It follows directly from Theorem \ref{equivrepresentation} and Lemma \ref{smallesteigenvalueoftwohoff}.
\end{proof}

\begin{example}\label{E6}
Note that the lattice $\Lambda(\widetilde{E}_6)$ is isomorphic to the lattice $E_6$, and hence the graph $\widetilde{E}_6$ is not $1$-integrable. Corollary \ref{twokindsofgraphs} tells us that there exists a positive integer $D$ such that the graph $G(\mathfrak{p}(\widetilde{E}_6),n)$ with smallest eigenvalue at least $-3$ is not $1$-integrable when $n\geq D$. As the minimal degree of $G(\mathfrak{p}(\widetilde{E}_6),n)$ is equal to $n$, we obtain an infinite family of graphs with smallest eigenvalue at least $-3$ and unbounded minimal degree.
\end{example}

\subsection{Associated Hoffman graphs}

In this subsection, we summarize some facts about \emph{associated Hoffman graphs} and \emph{quasi-cliques}, which provide some connection between Hoffman graphs and graphs. For more details, we refer to \cite{HJJ} and \cite{KYY}.

Let $m$ be a positive integer and let $G$ be a graph that does not contain $\widetilde{K}_{2m}$ as an induced subgraph, where $\widetilde{K}_{2m}$ is the graph on $2m+1$ vertices consisting of a complete graph $K_{2m}$ and a vertex $\infty$ which is adjacent to exactly $m$ vertices of the $K_{2m}$. Let $n\geq (m+1)^2 $ be a positive integer. Let $\mathcal{C}(n)$ := $\{C\mid$ $C$ is a maximal clique of $G$ with at least $n$ vertices$\}$. Define the relation $\equiv_n^m$ on $\mathcal{C}(n)$ by $C_1\equiv_n^m C_2$ if each vertex $x\in C_1$ has at most $m-1$ non-neighbors in $C_2$ and each vertex $y\in C_2$ has at most $m-1$ non-neighbors in $C_1$. Note that $\equiv_n^m$ is an equivalence relation.

Let $[C]_n^m$ denote the equivalence class of $\mathcal{C}(n)$ of $G$ under the equivalence relation $\equiv_n^m$ containing the maximal clique $C$ of $\mathcal{C}(n)$. We define the quasi-clique $Q([C]_n^m)$ of $C$ with respect to the pair $(m, n)$, as the induced subgraph of $G$ on the set $\{x\in V(G)\mid$ $x$ has at most $m-1$ non-neighbors in $C\}$.

Let $[C_1]_n^m, \dots, [C_r]_n^m$ be the equivalence classes of maximal cliques under $\equiv_n^m$. The \emph{associated Hoffman graph} $\go=\go(G,m,n)$ is the Hoffman graph satisfying the following conditions:
\begin{enumerate}
\item $V_s(\go) = V(G)$, $V_f(\go)=\{f_1, f_2, \dots, f_r\};$
\item The slim graph of $\go$ equals $G;$
\item For each $i$, the fat vertex $f_i$ is adjacent to exactly all the vertices of $Q([C_i]_n^m)$ for $i=1,2, \dots, r.$
\end{enumerate}

The following result, which is a crucial tool in this paper, is shown in \cite[Proposition 4.1]{HJJ}.

\begin{prop}\label{assohoff}
Let $G$ be a graph and let $m \geq 2, \phi,\sigma,p \geq 1$ be integers.
There exists a positive integer $n = n(m, \phi, \sigma,  p) \geq (m+1)^2$ such that for any integer $q \geq n$ and any Hoffman graph $\ho$ with at most $\phi$ fat vertices and at most $\sigma$ slim vertices, the graph  $G(\ho, p)$ is an induced subgraph of $G$, provided that the graph  $G$ satisfies the following conditions:
\begin{enumerate}
\item The  graph $G$ does not contain  $\widetilde{K}_{2m}$ as an induced subgraph;

\item Its associated Hoffman graph $\go = \mathfrak{g}(G, m, q)$ contains $\ho$ as an induced Hoffman subgraph.

\end{enumerate}
\end{prop}

\section{A limit result}

In this section, we will give a limit result on matrices.  As a first step, we need the following two inequalities which are direct consequences of the interlacing theorem and the Courant-Weyl inequalities.
We denote, for a symmetric real matrix $M$, its smallest eigenvalue by $\lambda_{\min}(M)$.

\begin{lemma}\label{interlacing}
\begin{enumerate}
\item (\cite[Theorem 9.1.1]{Ch&G}) If $M$ is a real symmetric matrix and $M^\prime$ a principal submatrix of $M$, then $\lambda_{\min}(M^\prime)\geq\lambda_{\min}(M)$.
\item (\cite[Theorem 2.8.1]{spectraofgraphs}) Let $M_1$ and $M_2$ be two real symmetric matrices. If $M:= M_1 -M_2$  is positive semidefinite, then $\lambda_{\min}(M_1)\geq\lambda_{\min}(M_2)$.
\end{enumerate}
\end{lemma}

Now we give the limit result. The main idea comes from the proof of \cite[Theorem 2.14]{HJAT}.

\begin{theorem}\label{Hoffman}
Let $\mathcal{I}$ be a finite set and $\pi = \{\mathcal{I}_1, \mathcal{I}_2\}$ a partition of $\mathcal{I}$. Assume $M := \bordermatrix{
 ~& \mathcal{I}_1& \mathcal{I}_2 \cr
\mathcal{I}_1& M^{1,1} & M^{1,2}\cr
\mathcal{I}_2& M^{2,1} & M^{2,2}\cr
 }$ is a block matrix indexed by $\mathcal{I}_1\cup\mathcal{I}_2$, which has smallest eigenvalue $\lambda_{\min}(M)$ at most $-1$. For a given non-empty set $\mathcal{I}_3$ disjoint from $\mathcal{I}$, say with $n \geq 1$ elements, define a symmetric block matrix $\widehat{M}(n)$ indexed by $\mathcal{I}\cup\mathcal{I}_3$ as follows:
$$\widehat{M}(n):=\bordermatrix{
 ~& \mathcal{I}_1& \mathcal{I}_2 & \mathcal{I}_3\cr
\mathcal{I}_1& M^{1,1} & M^{1,2}&O\cr
\mathcal{I}_2& M^{2,1} & M^{2,2}+J&J\cr
\mathcal{I}_3& O & J&J-I
 },$$
where $O$, $J$ and $I$ are the zero matrix, the all-ones matrix and the identity matrix, respectively. Denote by $\mu_n$ the smallest eigenvalue of $\widehat{M}(n)$. Then the following holds:
\begin{enumerate}
\item The sequence $(\mu_n)_{n \geq 1}$ is non-increasing;

\item $\mu_n \geq \lambda_{\min}(M)$ and $\lim\limits_{n\to\infty}{\mu_n}=\lambda_{\min}(M)$;

\item Let $\mu\leq-1$ be a real number. If $\widehat{M}(n)$ has a principal submatrix of order at most $m$ with smallest eigenvalue less than $\mu$, then $M$ also has a principal submatrix of order at most $m$ with smallest eigenvalue less than $\mu$.
\end{enumerate}
\end{theorem}
\begin{proof}

\begin{itemize}
\item[(\rm{i})] If $n_1 \geq n_2$, then $\widehat{M}(n_2)$ is a principal submatrix of $\widehat{M}(n_1)$. By Lemma \ref{interlacing} (\rm{i}), we have $\mu_{n_2} \geq \mu_{n_1}$ and hence $(\mu_n)_{n \geq 1}$ is non-increasing.

\item[(\rm{ii})] As $$\widehat{M}(n)=\begin{pmatrix}
M^{1,1} & M^{1,2}&O\\
M^{2,1} & M^{2,2}&O\\
O & O&-I
\end{pmatrix}+
\begin{pmatrix}
O & O&O\\
O& J&J\\
O & J&J
\end{pmatrix},$$
and $\begin{pmatrix}
O & O&O\\
O& J&J\\
O & J&J
\end{pmatrix}$ is positive semidefinite,  by Lemma \ref{interlacing} (\rm{ii}), we obtain

\begin{equation}\label{limit}
\mu_n\geq\lambda_{\min}(
\begin{pmatrix}
M^{1,1} & M^{1,2}&O\\
M^{2,1} & M^{2,2}&O\\
O & O&-I
\end{pmatrix})=\lambda_{\min}(M)\text{ for }n=1,2,\ldots
\end{equation}

Since $(\mu_n)_{n \geq 1}$ is a non-increasing sequence bounded below by $\lambda_{\min}(M)$, we see that $\lim\limits_{n\to\infty}{\mu_n}$ exists.

To show $\lambda_{\min}(M)\geq\lim\limits_{n\to\infty}{\mu_n}$, we will show that $$\lambda_{\min}(M)\geq\mu_n+\frac{(-\mu_n-1)|\mathcal{I}_2|}{n-\mu_n-1}~\text{ holds for all }n\geq2.$$
As $J-I$ is a principal submatrix of $\widehat{M}(n)$ and has smallest eigenvalue $-1$ if $n \geq 2$, we see that $\mu_n\leq-1$ for $n \geq 2$.
Since the matrix $\widehat{M}(n)-\mu_nI$ is positive semidefinite, there exist matrices $\widehat{N},\widehat{N}^{1,1},\widehat{N}^{1,2}$ and $\widehat{N}^{1,3}$ such that the following holds:
 \begin{enumerate}
\item $\widehat{N}=(\widehat{N}^{1,1},\widehat{N}^{1,2},\widehat{N}^{1,3})$;
 \item $\widehat{M}(n)-\mu_nI=\widehat{N}^T\widehat{N};$
\item $(\widehat{N}^{1,1})^T\widehat{N}^{1,1}=M^{1,1}-\mu_nI;$
\item $(\widehat{N}^{1,2})^T\widehat{N}^{1,2}=M^{2,2}+J-\mu_nI; $ and
\item $(\widehat{N}^{1,3})^T\widehat{N}^{1,3}=J-(\mu_n +1)I$.
\end{enumerate}
We will write $\mathbf{j}$ for a (column) vector with only ones.
Set $\epsilon_1=1-\sqrt{\frac{n}{n-\mu_n-1}}$, $\epsilon_2=\sqrt{\frac{-\mu_n-1}{n-\mu_n-1}}$, and $\mathbf{u}=\frac{1}{\sqrt{n(n-\mu_n-1)}}\widehat{N}^{1,3}\mathbf{j}$.   Then we have
\begin{equation}\label{eq1}
\mathbf{u}^T\mathbf{u}=1,~\mathbf{u}^T\widehat{N}^{1,1}=\mathbf{0},~\mathbf{u}^T\widehat{N}^{1,2}=(1-\epsilon_1)\mathbf{j}^T,~\text{and  }\epsilon^2_2=2\epsilon_1-\epsilon_1^2.
\end{equation}
Let $B$ be the $ \binom{|\mathcal{I}_2|}{2}\times|\mathcal{I}_2|$ matrix defined by $B_{\{i,j\},k}=\delta_{i,k}-\delta_{j,k}$, where $i,j,k\in \mathcal{I}_2$ and $i\neq j$. Then
\begin{equation}\label{eq2}
B^TB=|\mathcal{I}_2|I-J.
\end{equation}
Define the matrix $N$ as follows:
$$N=\begin{pmatrix}
\widehat{N}^{1,1} & \widehat{N}^{1,2}+(\epsilon_1-1)\mathbf{u}\mathbf{j}^T\\
\epsilon_2\sqrt{|\mathcal{I}_2|}I& O \\
O& \epsilon_2B
\end{pmatrix}.
$$
By (\ref{eq1}) and (\ref{eq2}), we find that $N^TN=M+(-\mu_n+\epsilon_2^2|\mathcal{I}_2|)I$ holds and this completes the proof.

\item[(\rm{iii})] Suppose $\widehat{M}^\prime(n)$ is a principal submatrix of $\widehat{M}(n)$ with order $m$ and smallest eigenvalue less than $\mu$, indexed by $\mathcal{J}_1 \cup \mathcal{J}_2\cup\mathcal{J}_3$, where $\mathcal{J}_i \subseteq \mathcal{I}_i$ for $i=1,2,3.$ Let $M^\prime$ be the principal submatrix of  $M$ indexed by $\mathcal{J}_1 \cup \mathcal{J}_2$. Then replacing $M$ and $\widehat{M}(n)$ in Inequality (\ref{limit}) by $M^\prime$ and $\widehat{M}^\prime(n)$, respectively, we see that $M^
\prime$ has smallest eigenvalue at most $\mu$.
\end{itemize}
\end{proof}
\section{Minimal forbidden fat Hoffman graphs}

Let $\Go_3$ be the set of fat Hoffman graphs with smallest eigenvalue at least $-3$. Then, by Lemma \ref{hoff}, for any $\ho \in \Go_3$, every fat induced Hoffman subgraph of $\ho$ is also contained in $\Go_3$. A {\it minimal forbidden fat Hoffman graph} for $\Go_3$ is a fat Hoffman graph $\mathfrak{f}\notin \Go_3$ such that every proper fat induced Hoffman subgraph of $\fo$ is contained in $\Go_3$.

In this section, we will show that the number of isomorphism classes of minimal forbidden fat Hoffman graphs for $\Go_3$ is finite and all these minimal forbidden fat Hoffman graphs have at most $10$ slim vertices. We first need a forbidden matrix result.

\subsection{A forbidden matrix result}
In this subsection, we will show a forbidden matrix result, which will be used later to show that any minimal forbidden fat Hoffman graph for $\Go_3$ has at most $10$ slim vertices.  We will use a result of Vijayakumar to show this result.

In \cite{signgraph1}, Vijayakumar proved that any connected signed graph with smallest eigenvalue less than $-2$ has an induced signed subgraph with at most 10 vertices and smallest eigenvalue less than $-2$.  As a consequence of this result, he showed the following:

\begin{theorem}(\cite[Theorem 4.2]{signgraph1})\label{Vi1} Let $M$ be a real symmetric matrix, whose diagonal entries are $0$ and off-diagonal entries are integers. If the smallest eigenvalue of $M$ is less than $-2$, then one of its principal submatrices of order at most $10$ also has an eigenvalue less than $-2$.

\end{theorem}

By Theorem \ref{Hoffman}, we obtain the following generalization of Theorem \ref{Vi1}.

\begin{theorem}\label{Vi2}
Let $M$ be a real symmetric matrix, whose diagonal entries are $0$ or $-1$ and off-diagonal entries are integers. If the smallest eigenvalue of $M$ is less than $-2$, then one of its principal submatrices of order at most $10$ also has an eigenvalue less than $-2$.
\end{theorem}
\begin{proof}
Let $\mathcal{I}_1:=\{i\mid M_{ii}=0\}$ and $\mathcal{I}_2:=\{i\mid M_{ii}=-1\}$ be two disjoint sets of indices. Without loss of generality, we may assume that $M$ is the block matrix $\bordermatrix{
 ~& \mathcal{I}_1& \mathcal{I}_2 \cr
\mathcal{I}_1& M^{1,1} & M^{1,2}\cr
\mathcal{I}_2& M^{2,1} & M^{2,2}\cr
 }$. Then the matrix $\widehat{M}(n)$, as defined in Theorem \ref{Hoffman}, has $0$-diagonal and integral off-diagonal entries. Hence the result follows by Theorem \ref{Vi1} and Theorem \ref{Hoffman}.

\end{proof}

\subsection{Forbidden special matrices}

In this subsection, we will give some necessary conditions on the matrices that can be the special matrices of minimal forbidden fat Hoffman graphs for $\Go_3$.

For two square matrices $M_1$ and $M_2$, we say that $M_1$ is {\it equivalent} to $M_2$, if there exists a permutation matrix $P$ such that $P^T M_1 P= M_2$. A square matrix $M$ is called {\it irreducible} if $M$ is not equivalent to a block diagonal matrix.

\begin{proposition}\label{sp}
If $\fo$ is a minimal forbidden fat Hoffman graph for $\Go_3$, then the matrix $M := Sp(\fo)+I$ satisfies the following properties.
\begin{enumerate}
 \item $M$ is symmetric and irreducible;
 \item $M_{ij}\in\mathbb{Z}$, $M_{ij} \leq 1$, and $M_{ii}\leq0$  for all $i$, $j$;
 \item $M_{ij} \geq \max\{M_{ii}, M_{jj}\}-1$  for all $i$, $j$;
 \item The smallest eigenvalue of $M$ is less than $-2$;
 \item Every principal submatrix of $M$ has smallest eigenvalue at least $-2$.
\end{enumerate}

\end{proposition}

\begin{proof}

Note that the columns and the rows of $Sp(\fo)$ are indexed by the slim vertices of $\fo$ and the smallest eigenvalue of $M$ is less than $-2$.

Since $Sp(\fo)$ is symmetric, $M$ is also symmetric. Let $U$ be a proper subset of $V_s(\fo)$ and let $\fo'$ be the Hoffman subgraph generated by $U$. Then $Sp(\fo') +I$ is exactly the principal submatrix indexed by $U$. This shows (v) and that $M$ is irreducible.





By Lemma \ref{fatnbr}, (ii) and (iii) hold. It is clear that (iv) holds.

\end{proof}

Let $\widehat{\mathcal{M}}$ be the set of square matrices containing all matrices $M \in \widehat{\mathcal{M}}$ satisfying the five properties of Proposition \ref{sp}.

\begin{proposition} \label{forbidden}
Let $M \in \widehat{\mathcal{M}}$ be a matrix.
\begin{enumerate}

\item If $M$ contains a diagonal entry at most $-3$, then $M = M_1(a)$ for some integer $a \geq 3$ where $M_{1}(a)= \begin{pmatrix}
-a
\end{pmatrix}$;

\item If $M$ contains a diagonal entry $-2$, then $M$ is equivalent to $M_2(b_1, b_2)$ for some integers  $b_1 \in \{1, -1, -2, -3\}$,  and $-2 \leq b_2 \leq \min\{0,b_1 +1\}$, where $M_{2}(b_1, b_2)=\begin{pmatrix}
-2 & b_1 \\
b_1 & b_2
\end{pmatrix}$;

\item If all diagonal entries of $M$ are $0$ or $-1$, then $M$ has order at most $10$.






\end{enumerate}
\end{proposition}

\begin{proof} Statements (i) and (ii) are straightforward to check.
Statement (iii) follows directly from Theorem \ref{Vi2}.
  \end{proof}

Now we are in the position to show that there are only finitely many (isomorphism classes of) minimal forbidden fat Hoffman graphs for $\Go_3$.

\begin{proposition}\label{forbiddenhoff}
Let $\fo$ be a minimal forbidden fat Hoffman graph for $\Go_3$. Then $\fo$ satisfies exactly one of the following statements:

\begin{enumerate}

\item There exists a slim vertex of $\fo$ which contains at least four fat neighbors. In this case, $\fo$ is isomorphic to

$$\raisebox{-0.35ex}{\includegraphics[scale=0.28]{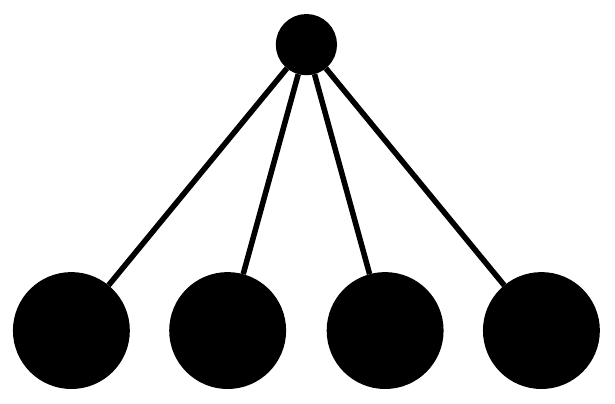}}.$$

\item Every slim vertex of $\fo$ has at most three fat neighbors. In this case, $\fo$ has at most $10$ slim vertices.

\end{enumerate}

In particular, there are finitely many isomorphism classes of minimal forbidden fat Hoffman graphs for $\Go_3$.

\end{proposition}

\begin{proof}

Suppose that there exists a slim vertex of $\fo$ with at least four fat neighbors. Then $Sp(\fo) + I$ is equivalent to $M_1(a)$ for some $a \geq 3$. If $a$ is at least $4$, then $\fo$ contains \raisebox{-0.3ex}{\includegraphics[scale=0.13]{m4}} as a proper induced Hoffman subgraph. By the minimality of $\fo$, we have $a=3$. Thus $Sp(\fo) = \begin{pmatrix} -4 \end{pmatrix}$ and $\fo$ is isomorphic to \raisebox{-0.3ex}{\includegraphics[scale=0.13]{m4}}.

Now, suppose that every slim vertex of $\fo$ has at most three fat neighbors. Then $Sp(\fo)+I$ has order at most $10$ by Proposition \ref{forbidden} (ii) and (iii). This shows that $\fo$ has at most $10$ slim vertices.

\end{proof}

\section{Proof of the main theorem}

In this section, we will prove Theorem \ref{mainthm}.

\begin{lemma}\label{main1}
There exists a positive integer $p$ such that if G is a connected graph with smallest eigenvalue $\lambda_{\min}(G)$ at least $-3$, then the associated Hoffman graph $\go(G,12,r)$ has smallest eigenvalue at least $-3$  for any integer $r\ge p$.
\end{lemma}
\begin{proof}
Let $G$ be a graph with $\lambda_{\min}(G)\geq-3$. As $\lambda_{\min}(\widetilde{K}_{24})<-3$, it follows that $G$ does not contain  $\widetilde{K}_{24}$ as an induced
subgraph. Now it suffices to show that $\go(G,12,r)$ does not contain an isomorphic copy of a minimal forbidden fat Hoffman graph for $\Go_3$ as an induced Hoffman subgraph.
Let $\{\fo_1,\fo_2,\dots,\fo_{s}\}$ be the set of the isomorphism classes of the minimal forbidden fat Hoffman graphs for $\Go_3$. By Theorem \ref{Ostrowski}, we find that for
all $1\leq i \leq s$, there exists a positive integer $p_i^\prime$ such that $\lambda_{\min}(G(\fo_i, p_i^\prime))<-3$, since $\lambda_{\min}(\fo_i)<-3$. Let $p_i=n(12,
|V_f(\mathfrak{h}_i)|,|V_s(\mathfrak{h}_i)|,p_i^\prime)$ and $p = \max_{1\leq i \leq s}p_i$. Then for $r\geq p$, the Hoffman graph $\go(G,12,r)$ does not contain an isomorphic copy of any $\fo_i$ as an induced Hoffman subgraph by Proposition \ref{assohoff}. This completes the proof.
\end{proof}

Now we recall the Ramsey theorem.

\begin{theorem}\label{Ramsey}
(\cite{Ramsey}) Let a, b be two positive integers. Then there exists a minimum positive integer $R(a,b)$ such that for any graph $G$ on $n\geq R(a,b)$ vertices, the graph $G$ contains a dependent set of size $a$ or an independent set of size $b$.
\end{theorem}
The number $R(a,b)$ in the previous theorem is called a Ramsey number.

With Lemma \ref{main1} and Theorem \ref{Ramsey}, we will show the following theorem.
\begin{theorem}\label{main2}
There exists a positive integer $K$, such that if a graph $G$ with smallest eigenvalue at least $-3$ has minimal degree at least $K$, then $G$ is the slim graph of a fat Hoffman graph with smallest eigenvalue at least $-3$.
\end{theorem}

\begin{proof}
Choose $n=p$, where $p$ is such that Lemma \ref{main1} holds. Let $K $ be equal to the Ramsey number $R(n,10)$. For any vertex $x$ of $G$, the neighborhood of $x$ contains a clique with order at least $n$, since $G$ does not contain a $10$-claw as an induced subgraph. This implies that the associated Hoffman graph $\go(G,12,n)$ is fat. By Lemma \ref{main1}, the fat Hoffman graph $\go(G,12,n)$ has smallest eigenvalue at least $-3$ and $G$ is the slim graph of $\go(G,12,n)$.
\end{proof}

Now we are in the position to prove our main result Theorem \ref{mainthm}.
\begin{proof}[{\textit Proof of Theorem \ref{mainthm}}] By using Theorem \ref{cameron}, the theorem holds for graphs with smallest eigenvalue at least $-2$. Now we may assume that $-3\leq\lambda_{\min}(G)<-2$ holds. Let $K$ be a positive integer such that Theorem \ref{main2} holds. Then there exists a fat Hoffman graph $\mathfrak{h}$ with smallest eigenvalue at least $-3$ which has $G$ as the slim graph. Note that $\lfloor\lambda_{\min}(G)\rfloor=\lfloor\lambda_{\min}(\mathfrak{h})\rfloor=-3$. By Lemma \ref{twolattices} and Corollary \ref{graphandhoffman}, it is sufficient to show that the integral lattice $\Lambda^{red}(\frak{h},3)$ is $2$-integrable.

Suppose $N$ is a matrix such that $Sp(\mathfrak{h})+3I=N^TN$ holds. Denote by $N_{r_1},~N_{r_2},~\ldots,~N_{r_n}$ all the non-zero columns of $N$. Since the Hoffman graph $\mathfrak{h}$ is fat, then the norm $\|N_{r_i}\|^2$ of $N_{r_i}$ is $1$ or $2$ for $i=1,\ldots,n$. This means that the lattice $\Lambda^{red}(\frak{h},3)$ is a direct sum of a standard lattice and  irreducible root lattices. As any irreducible root lattice is 2-integrable, we find that $\Lambda^{red}(\frak{h},3)$ is also $2$-integrable. This completes the proof.
\end{proof}
\begin{remark} In Remark \ref{rem1} $\mathrm{(ii)}$, we claimed that there exists a real number  $\varepsilon<-3$, such that for any real number $\lambda\in(\varepsilon,-3]$, there exists a constant $f(\lambda)$, such that if a graph $G$ has smallest eigenvalue at least $\lambda$ and minimal degree at least $f(\lambda)$, then the smallest eigenvalue of $G$ is at least $-3$. The real number $\varepsilon$ can be found as the largest number among the smallest eigenvalues of the minimal forbidden fat Hoffman graphs for $\Go_3$. As there are finitely many isomorphism classes of the minimal forbidden fat Hoffman graphs for $\Go_3$, the number $\varepsilon$ exists.
\end{remark}

We finish this paper with two conjectures and one problem.
\begin{conjecture} There exists a positive integer $\sigma$ such that any connected graph $G$ with smallest eigenvalue at least $-3$ is $\sigma$-integrable.
\end{conjecture}
\begin{conjecture} There exists a positive integer $\sigma$ such that for any irreducible integral lattice $\Lambda$ generated by vectors of norm $3$, the lattice $\Lambda$ is $\sigma$-integrable.
\end{conjecture}

We would also like to formulate the following problem: Find more maximal connected graphs with smallest eigenvalue at least $-3$.

\section*{Acknowledgments}
J.H. Koolen is partially supported by the National Natural Science Foundation of China (No. 11471009 and No. 11671376).


\begin{thebibliography}{99}
\bibitem{drg} A.E. Brouwer, A.M. Cohen, A. Neumaier, \emph{Distance-Regular Graphs}, Springer-Verlag, Berlin, 1989.

\bibitem{spectraofgraphs} A.E. Brouwer, W.H. Haemers, \emph{Spectra of graphs}, Springer, New York, 2012.

\bibitem{Cameron} P.J. Cameron, J.M. Goethals, J.J. Seidel, E.E. Shult, Line graphs, root systems, and elliptic geometry, \emph{J. Algebra}, 43:305--327, 1976.


\bibitem{low dimension V} J.H. Conway, N.J.A. Sloane, Low-dimensional lattices V. Integral coordinates for integral lattices, \emph{Proc. Roy. Soc. London Ser. A}, 426:211-232, 1989.


\bibitem{code} W. Ebeling, Lattices and codes. A course partially based on lectures by F. Hirzebruch. Advanced Lectures in Mathematics. \emph{Friedr. Vieweg \& Sohn, Braunschweig}, 1994.

\bibitem{Mclaughlin} J.M. Goethals, J.J. Seidel, The regular two-graph on $276$ vertices, \emph{Discrete Math.}, 12:143--158, 1975.

\bibitem{Ch&G} C.D. Godsil, G. Royle, \emph{Algebraic Graph Theory}, Springer-Verlag, New York, 2001.


\bibitem{Hoff1977} A.J. Hoffman, On graphs whose least eigenvalue exceeds $-1-\sqrt2$, \emph{Linear Algebra Appl.}, 16:153--165, 1977.

\bibitem{HJAT} H.J. Jang, J.H. Koolen, A. Munemasa, T. Taniguchi, {On fat Hoffman graphs with smallest eigenvalue at least $-3$}, \emph{Ars Math. Contemp.}, 7:105--121, 2014.

\bibitem{HJJ} H.K. Kim, J.H. Koolen, J.Y. Yang, {A structure theory for graphs with fixed smallest eigenvalue}, \emph{Linear Algebra Appl.}, 540:1--13, 2016.

\bibitem{KRY2} J.H. Koolen, M.U. Rehman, Q. Yang, On the integrability of strongly regular graphs, preprint.

\bibitem{KYY} J.H. Koolen, J.Y. Yang, Q. Yang, A generalization of a theorem of Hoffman, arXiv:1612.07085.

\bibitem{odd unimodular lattices} C. Bachoc, G. Nebe, B. Venkov, Odd unimodular lattices of minimum $4$, \emph{Acta Arith.}, 101:151--158, 2002.

\bibitem{desigh} G. Nebe, B. Venkov, On lattices whose minimal vectors form a $6$-design, \emph{European J. Combin.}, 30:716--724, 2009.


\bibitem{Ramsey} F.P. Ramsey, {On a problem of formal logic}, \emph{Proc. London Math. Soc.}, 30:264--286, 1930.

\bibitem{Ernest} E. Shult, A. Yanushka, Near $n$-gons and line systems, \emph{Geom. Dedicata}, 9::1--72, 1980.

\bibitem{signgraph1} G.R. Vijayakumar, Signed graphs represented by $D_{\infty}$, \emph{European J. Combin.}, 8:103--112, 1987.

\bibitem{Witt} E. Witt, Spiegelungsgruppen und Aufz\"{a}hlung halbeinfacher Liescher Ringe, \emph{Abh. Math. Sem. Hansischen Univ.}, 14:289--322, 1941.

\bibitem{Woo} R. Woo, A. Neumaier, {On graphs whose smallest eigenvalue is at least $-1-\sqrt2$}, \emph{Linear Algebra Appl.}, 226--228:577--591, 1995.
\end{thebibliography}
\end{document}